\documentclass[leqno]{article}

\usepackage{layout}
\usepackage[final]{pdfpages}

\usepackage[OT1]{fontenc}
\usepackage{color}
\usepackage{amsthm,amsmath,graphicx,latexsym,amssymb,amscd,amsfonts,enumerate, bm}
\usepackage[colorlinks,citecolor=blue,linkcolor=red,urlcolor=blue]{hyperref}
\usepackage{authblk} 

\theoremstyle{plain}
\newtheorem{theorem}{Theorem}
\newtheorem{proposition}[theorem]{Proposition}

\newtheorem{remark}[theorem]{Remark}

\begin{document}

\title{\textbf{\Large Parameter estimation for one-sided heavy-tailed distributions}}

\date{ }

\author{Phillip Kerger\thanks{Department of Applied Mathematics and Statistics, Johns Hopkins University. Email:\ pkerger@jhu.edu} \ and Kei Kobayashi\thanks{Corresponding author. Department of Mathematics, Fordham University. Email:\ kkobayashi5@fordham.edu}}

\maketitle

%
%

\begin{abstract}
Stable subordinators, and more general subordinators possessing power law probability tails, have been widely used in the context of subdiffusions, where particles get trapped or immobile in a number of time periods, called constant periods. The lengths of the constant periods follow a one-sided distribution which involves a parameter between 0 and 1 and whose first moment does not exist. 
This paper constructs an estimator for the parameter, applying the method of moments to the number of observed constant periods in a fixed time interval.
The resulting estimator is asymptotically unbiased and consistent, and it is well-suited for situations where multiple observations of the same subdiffusion process are available. 
We present supporting numerical examples and an application to market price data for a low-volume stock. 
 \vspace{3mm}

\vspace{3mm}


\end{abstract}

\section{Introduction}\label{section_1}

A wide variety of phenomena whose states change over time at random exhibit a number of time periods in which the states remain unchanged.
These range across stock market data, diffusion of large molecules within cells, movement of bacteria, electricity markets, and more; see e.g.\ Chapter 1 of \cite{HKU-book}. 
For example, in finance, although the celebrated Black--Scholes--Merton model has been widely used to describe fluctuations of the prices of financial products, the model uses
a geometric Brownian motion, which fails to capture the constant-price periods that many low-volume equities experience.
This motivates the construction of time-changed stochastic processes exhibiting such constant periods. 
As a simple example, consider a Brownian motion $(B_t)_{t\ge 0}$ composed with the inverse $(E_t)_{t\ge 0}$ of an independent $\beta$-stable subordinator (or \textit{inverse $\beta$-stable subordinator} for short). The resulting stochastic process $B\circ E=(B_{E_t})_{t\ge 0}$, called a \textit{time-changed Brownian motion}, shows constant periods and has variance growing at the rate of $t^\beta$. Since the stability index $\beta$ takes a value in the interval $(0,1)$ and the variance of the Brownian motion $(B_t)$ grows at the rate of $t$, particles represented by the time-changed Brownian motion spread at a slower rate than the regular Brownian particles for $t>1$. Because of this, the particular time-changed Brownian motion $(B_{E_t})$ and its variants are often referred to as \textit{subdiffusion processes}.  
Figure \ref{fig:three_figures} gives a graphical comparison of sample paths of the inverse $\beta$-stable subordinator $(E_t)$ and the time-changed Brownian motion $(B_{E_t})$ with different values of $\beta$. 
As the simulations show, 
we expect to observe longer constant periods for a smaller value of $\beta$.
 The time-changed Brownian motion $(B_{E_t})$ is not a Markov process and  
 its densities satisfy the time-fractional heat equation 
$
    \partial_t^\beta p(t,x)=(1/2) \partial_x^2 p(t,x), 
$
where $\partial_t^\beta$ denotes the Caputo fractional derivative of order $\beta$ (see \cite{MS_1,MS_2}). 
For more details about properties of $(B_{E_t})$ and its various extensions (including stochastic differential equations) as well as their continuous-time random walk counterparts, see e.g.\ \cite{MeerschaertSikorskii,HKU-book} and references therein.

\begin{figure}\centering
\begin{minipage}{0.62\hsize}
  \includegraphics[width=1.8in,height=1.8in]{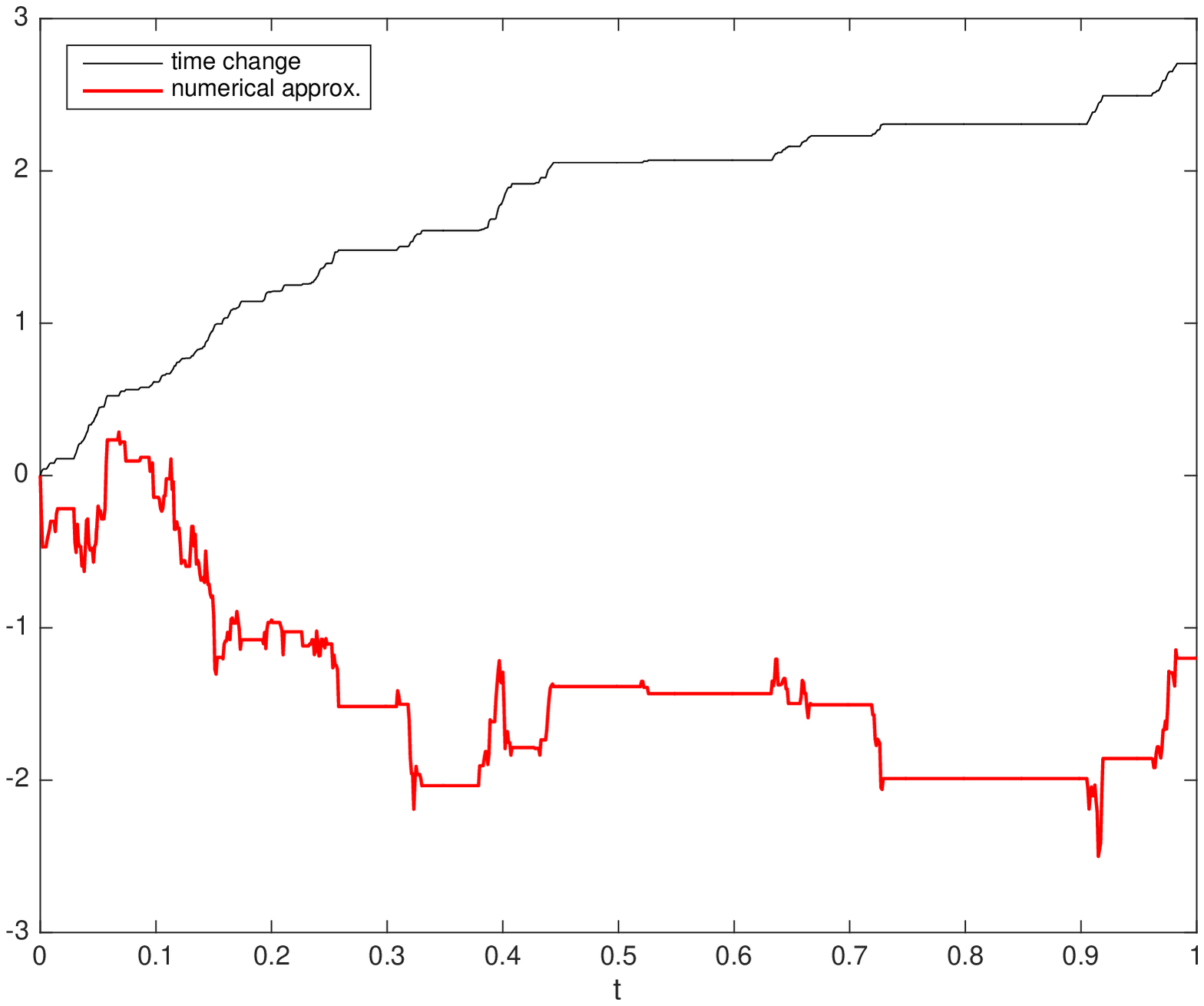}
  \includegraphics[width=1.8in,height=1.8in]{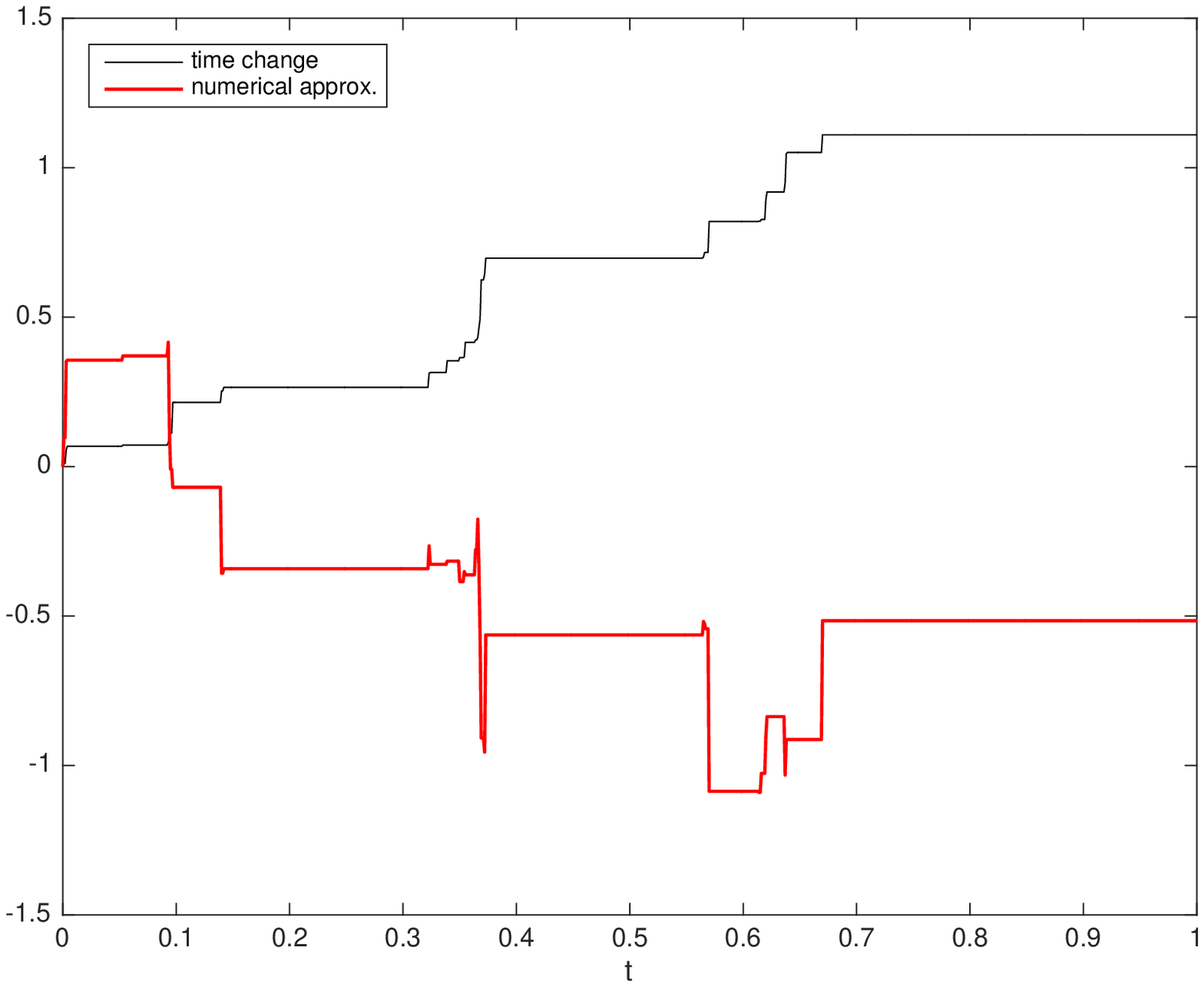}
  \caption{Sample paths of the inverse $\beta$-stable subordinator $(E_t)$ (black) and the time-changed Brownian motion $(B_{E_t})$ (red) with  
  $\beta=0.7$ (left) and $\beta=0.5$ (right)}
   \label{fig:three_figures}
\end{minipage}
\hspace{5mm}
\begin{minipage}{0.33\hsize}
  \includegraphics[height=4.5cm, width = 5cm]{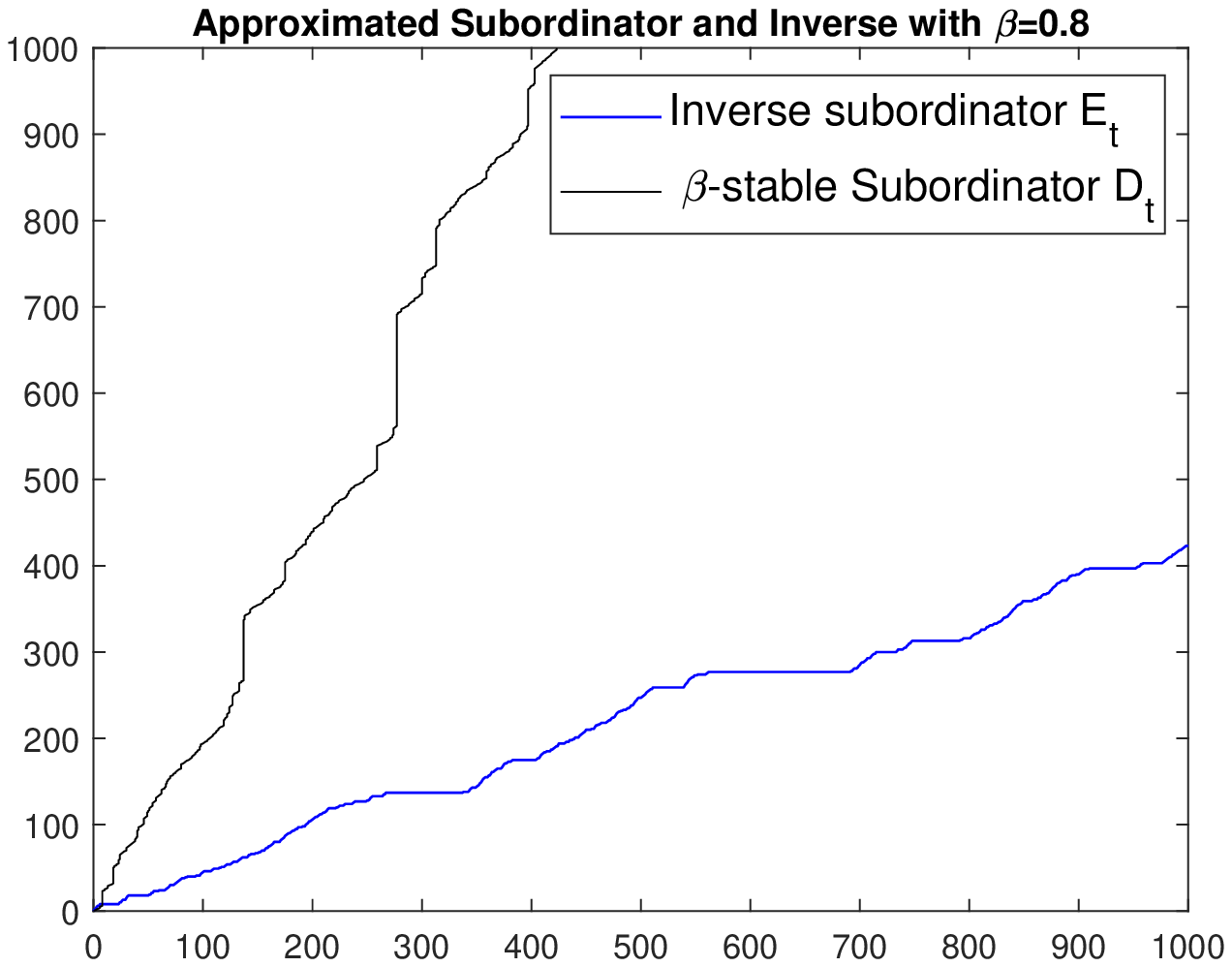}
  \caption{Sample paths of a $0.8$-stable subordinator $(D_t)$ (black) and its inverse $(E_t)$ (blue).} 
  \label{fig:inverse}
\end{minipage}

\end{figure}

Various methods for estimating the stability index of a stable distribution 
have been presented in the literature; see e.g.\ \cite{AbanMeerschaert01,AbanMeerschaert04,Beirlant99,Cahoy2012,Feuerverger99,Hill75,HoskingWallis87,McCulloch97,MS98,Nolan98}.
For example, Hill's estimator in \cite{Hill75} is popular and well-known for its robustness; it only assumes the underlying distribution has power law tails. 
In \cite{AbanMeerschaert01}, the authors proposed a shift-invariant version of Hill's estimator to increase the robustness. On the other hand, in  \cite{AbanMeerschaert04}, they revealed a connection between Hill's estimator and a least-squares estimator obtained from a $\log$-$\log$ transformation of the power law tails. 
The maximum likelihood estimator (MLE)  
was established in \cite{Nolan98}, overcoming the difficulty that the stable density does not have a closed form. 
The MLE in general requires a particular distributional form, so it is not as robust as Hill's estimator; however, it uses all of the data, rather than only the largest order statistics based on which Hill's estimator and its variants are constructed. 
In \cite{Cahoy2012}, another estimator using all of the data was constructed via the analysis of the $M$-wright function together with the method of moments.

In \cite{Janczura}, the authors compared six known parameter estimation methods in the context of subdiffusions, where the lengths of the constant periods follow a one-sided (totally right-skewed) $\beta$-stable distribution. 
This is the situation we also consider in this paper. However, we propose a new estimator for the stability index $\beta$ based on the recent development of numerical approximations of subdiffusion processes in \cite{JinKobayashi,JumKobayashi,Magdziarz_simulation,Magdziarz_spa}. Our idea is to regard real data exhibiting constant periods as a numerically approximated path of some time-changed process 
and to use the method of moments to estimate $\beta$. Note, however, that the method of moments does not apply to the stable distribution itself in the usual sense since it does not have the first moment. We instead apply the method to the ``number'' (rather than the ``lengths'') of observed constant periods. 
Cahoy's estimator in \cite{Cahoy2012} is based on the logarithmic moments of the stable distribution and has a connection to the estimator to be proposed in this paper; see Remark \ref{remark_Tlarge}.

Our estimator is especially suited for use on data sets with many observations over a long period of time (i.e.\ when both the number of observed paths and the time horizon are large) since it is asymptotically unbiased and consistent; see Remark \ref{remark_unbiasedness}(a). 
Moreover, it is robust in the sense that it only requires the underlying subordinator to have power law probability tails; see Remark \ref{remark_mixture}.
One major advantage of using our estimator as opposed to some other well-known estimators is that it allows one to estimate the entire range of $\beta$ values with reasonable accuracy and precision; see Remark \ref{remark_advantage}(a).

\section{Stable subordinators and their inverses}\label{section_2}

Throughout the paper, all stochastic processes are assumed to be defined on a probability space $(\Omega, \mathcal{F}, \mathbb{P})$ and $\mathbb{E}$ denotes the expectation under $\mathbb{P}$. 
A stochastic process $(D_t)_{t\ge 0}$ having independent and stationary increments is called a \textit{L\'evy process}. For the purpose of this paper, let us consider a one-dimensional L\'evy process with increasing, c\`adl\`ag paths (right continuous paths with left limits) starting at $0$, which is usually called a \textit{subordinator}. The distribution of a subordinator is characterized by its Laplace transform 
\begin{align}\label{Laplace_stable}
    \mathbb{E}[e^{-sD_t}]=e^{-t\psi(s)}, \ \ s>0, 
\end{align}
where $\psi(s)$, called the \textit{Laplace exponent} of $(D_t)$, is a Bernstein function on $(0,\infty)$ with $\psi(0^+)=0$.  
A \textit{$\beta$-stable subordinator} is a subordinator with $\psi(s)=s^\beta$ for $s>0$, 
where $\beta\in(0,1)$ is a parameter called the \textit{(stability) index}.

The probability tails of the $\beta$-stable subordinator $(D_t)$ are very different from those of a Brownian motion $(B_t)$. 
Namely, for $t>0$, while $B_t$ shows the exponential decay 
\[
	\mathbb{P}(B_t \ge x)\sim \left(\frac{t}{2\pi}\right)^{1/2} x^{-1}e^{-x^2/(2t)} \ \ \textrm{as} \ \  x\to\infty, 
\]
$D_t$ exhibits the slower, polynomial decay 
\[
	\mathbb{P}(D_t \ge x)\sim \frac{t}{\Gamma(1-\beta)}x^{-\beta}  \ \ \textrm{as} \ \  x\to\infty,
\] 
where $\Gamma(\cdot)$ is the Gamma function (see p.37 and p.76 of \cite{Applebaum}). 
Due to the presence of the heavy tail, $D_t$ does not have the $p$th moment for $p\ge \beta$, so the method of moments is not applicable in the usual sense. 
The stable subordinator $(D_t)$ has strictly increasing paths that go to $\infty$ as $t\to\infty$ with infinitely many jumps, where the jump times are dense in $(0,\infty)$ (see Theorem 21.3 of \cite{Sato}).

Define the \textit{inverse} or \textit{first hitting time process} of a $\beta$-stable subordinator $(D_t)$ by 
\[
	E_t=\inf\{u>0: D_u>t\}
\]
 for $t\ge 0$. 
The process $(E_t)_{t\ge 0}$ is called an \textit{inverse $\beta$-stable subordinator}. 
Since $(D_t)$ has strictly increasing paths starting at 0, $(E_t)$ has continuous, nondecreasing paths starting at 0. 
Figure \ref{fig:inverse} illustrates the inverse relationship between the sample paths of $(D_t)$ and $(E_t)$. 
The jumps of $(D_t)$ correspond to the constant periods of $(E_t)$, and the inverse relation implies 
\[
    \{E_t\le x\} = \{D_x\ge t\}  \ \ \textrm{for all} \ \ t,x\ge 0. 
\]
Moreover, unlike $D_t$, the random variable $E_t$ possesses finite exponential moments, i.e.\ $\mathbb{E}[e^{\lambda E_t}]<\infty$ for any $\lambda>0$ (see e.g.\  \cite{JumKobayashi,MagdziarzOrzelWeron}). In particular, the following formula holds for the $p$th moment of $E_t$ for any $p>0$ (see e.g.\ Proposition 5.6 of \cite{HKU-book} or Example 3.1 of \cite{JumKobayashi}) and plays a key role in deriving an estimator for the index $\beta$:
\begin{align}\label{first_moment}
    \mathbb{E}[(E_t)^p]=\frac{\Gamma(p+1)}{\Gamma(p\beta+1)}t^{p\beta}. 
\end{align}

\section{Derivation of an Estimator and its Properties}\label{section_3}

To derive an estimator for the index $\beta$ of an inverse stable subordinator $(E_t)$,  
we briefly discuss below the approximation scheme of $(E_t)$ first presented in \cite{Magdziarz_simulation,Magdziarz_spa} and subsequently used in \cite{JinKobayashi,JumKobayashi}.

Fix an equidistant step size $\delta>0$ and a time horizon $T>0$.
Simulate a sample path of the stable subordinator $(D_t)$ of index $\beta\in(0,1)$,
which has independent and stationary increments, 
by setting $D_0=0$ and then following the rule 
$
    D_{i\delta}:=D_{(i-1)\delta}+Z_i,
$
for $i=1,2,3,\ldots,$
where $Z_i$'s are i.i.d.\ random variables having the same distribution as $D_{\delta}$ (which can be generated via an algorithm presented in \cite{WeronR}). Stop this procedure upon finding the integer $K$ satisfying 
	$T\in[D_{K\delta}, D_{(K+1)\delta}).$
For each path of $(D_t)$, the number $K$ exists as a finite number since $D_t\to\infty$ as $t\to\infty$.   
Note that $K$ is an $\mathbb{N}\cup \{0\}$-valued random variable depending on the initially chosen constants $\delta$ and $T$. 

Next, let
\[
	E^\delta_t
	:=(\min\{n\in \mathbb{N}; D_{n\delta}>t\}-1)\delta
\]
for $t\in[0,T]$.
The sample paths of $(E^\delta_t)_{t\ge 0}$ are nondecreasing step functions with constant jump size $\delta$ and the length of the $i$th constant time interval given by $Z_i$. 
Indeed, 
$
	E^\delta_t=n\delta
$
whenever $t\in[D_{n\delta},D_{(n+1)\delta})$. 
In particular, each path of $(E^\delta_t)$ has a total of $(K+1)$ constant periods, and a.s.,
\begin{align}\label{Edelta_N}
    E^\delta_T=K\delta.
\end{align}
Moreover, it has been derived in \cite{JumKobayashi,Magdziarz_spa} that a.s.,
\begin{align}\label{ineq_Spsi}
	E_t-\delta\le E^\delta_t\le E_t \ \ \textrm{for all} \ \ t\in[0,T].
\end{align}
 Therefore, a.s., $(E^\delta_t)$ converges to  $(E_t)$ uniformly on the time interval $[0,T]$ as $\delta\downarrow 0$. Moreover, 
$
    \mathbb{E}[E_T]-\delta\le \mathbb{E}[E^\delta_T]\le \mathbb{E}[E_T].
$
This together with \eqref{Edelta_N} and \eqref{first_moment} with $p=1$ gives 
\begin{align}\label{key_estimate}
    \frac{T^\beta}{\Gamma(\beta+1)}-\delta
    \le \delta\mathbb{E}[K]\le \frac{T^\beta}{\Gamma(\beta+1)}, 
    \ \ \textrm{or equivalently,} \ \ 
    \eta(\beta)-1
    \le \mu_K\le \eta(\beta),
\end{align}
where $\mu_K=\mathbb{E}[K]$ and 
\begin{align}\label{k_function}
    \eta(\beta)
    =\dfrac{T^{\beta}}{\delta\Gamma(\beta+1)}.
\end{align}
Note that even though $\mu_K$ and $\eta(\beta)$ do not coincide, they are asymptotically equivalent as $T\to\infty$; i.e.\ for fixed $\beta\in(0,1)$ and $\delta>0$, 
\begin{align}\label{revision_1}
   \mu_K\sim \eta(\beta) \ \  \textrm{as} \ \  T\to\infty.
\end{align}
Based on this observation, we define a \textit{method-of-moments-like estimator} $\hat{\beta}$ for the parameter $\beta$ implicitly as a solution to the equation
\begin{align}\label{estimator_def}
  \bar{K} = \eta(\hat{\beta}), 
\end{align}
where $\bar{K}:=\sum_{i=1}^n K_i/n$ is the sample mean of a random sample $K_1,K_2,\ldots,K_n$ of size $n \in \mathbb{N}$ from the distribution of the random variable $K$. 
This is not the method of moments in the usual sense as we do not have the equality $\mu_K=\eta(\beta)$ for a fixed time horizon $T$; however, the relation \eqref{revision_1} makes the definition of our estimator $\hat{\beta}$ reasonable when $T$ is sufficiently large. In the remainder of this section, we discuss how large $T$ should be as well as the properties of the estimator $\hat{\beta}$. First, as the following proposition shows, $\hat{\beta}$ satisfying \eqref{estimator_def}
uniquely exists as long as $\bar{K}\in(1/\delta,T/\delta)$ and $T>e^{1-\gamma}\approx 1.5262$, where $\gamma\approx 0.5772$ is the Euler--Mascheroni constant.

\begin{proposition}\label{proposition_monotone}
Let $\delta>0$. 
If $T > e^{1-\gamma}\approx 1.5262$, then the function $\eta(\beta)$ in \eqref{k_function} is a smooth, strictly increasing bijection from $(0,1)$ to $(1/\delta,T/\delta)$. 
Moreover, if $T> e^{1-\gamma+\sqrt{\pi^2/6}}\approx 5.5032$, then $\eta(\beta)$ is convex on $(0,1)$. 
\end{proposition}

\begin{proof}
The smoothness of the function $\eta$ follows by the smoothness of the Gamma function. Suppose $T > e^{1-\gamma}$ and observe that
\begin{align}\label{eta_prime}
    \eta'(\beta) &= \dfrac{\Gamma(\beta+1)T^{\beta}\log T - T^{\beta}\Gamma'(\beta+1)}{\delta\Gamma^2(\beta+1)}
    =\dfrac{T^{\beta}\left(\log T - \Psi_0(\beta+1)\right)}{\delta\Gamma(\beta+1)},
\end{align}
where $\Psi_0$ is the digamma function defined by 
$
    \Psi_0(x)=(\mathrm{d}/\mathrm{d}x)\log\Gamma(x)=\Gamma'(x)/\Gamma(x)
$
for $x>0$. 
The digamma function and its derivative have the series representations 
$
    \Psi_0(x)=\sum_{k=0}^\infty ((k+1)^{-1}-(k+x)^{-1})-\gamma
$
and 
$
    \Psi_0'(x)=\sum_{k=0}^\infty (k+x)^{-2}
$
(see 6.3.16 and 6.4.10 in \cite{Abramowitz}).
Since $\Psi_0$ is strictly increasing, 
    $\Psi_0(\beta+1) < \Psi_0(2) = 1 - \gamma<\log T$ for any $\beta\in(0,1)$. 
Combining this with \eqref{eta_prime} 
yields $\eta'(\beta)>0$ for all $\beta\in(0,1)$. Since  $\eta(0^+)=1/\delta$ and $\eta(1^-)=T/\delta$, it follows that $\eta$ is a strictly increasing bijection from $(0,1)$ to $(1/\delta,T/\delta)$.

Now, suppose $T>e^{1-\gamma+\sqrt{\pi^2/6}}$. To prove the convexity of $\eta$, observe that
\begin{align*}
	\eta''(\beta)
	=\frac{T^\beta}{\delta\Gamma(\beta+1)}\Bigl(\log T-\Psi_0(\beta+1)+\sqrt{\Psi_0'(\beta+1)}\Bigr)\Bigl(\log T-\Psi_0(\beta+1)-\sqrt{\Psi_0'(\beta+1)}\Bigr). 
\end{align*}
Thus, the convexity follows upon verifying that $\log T-\Psi_0(\beta+1)-\sqrt{\Psi_0'(\beta+1)}>0$ for all $\beta\in(0,1)$. Since $\Psi_0$ is increasing and $\Psi_0'$ is decreasing, 
\[
    \log T-\Psi_0(\beta+1)-\sqrt{\Psi_0'(\beta+1)}
    >\log T-\Psi_0(2)-\sqrt{\Psi_0'(1)}
    =\log T-(1-\gamma)-\sqrt{\pi^2/6}.
  \]
The latter is positive since $T> e^{1-\gamma+\sqrt{\pi^2/6}}$, which completes the proof.
\end{proof}

In the remainder of the section, assume that $T > e^{1-\gamma}$. 
We have so far defined the method-of-moments-like estimator $\hat{\beta}$ only conditionally on the event that $\bar{K}\in(1/\delta,T/\delta)$. 
To define $\hat{\beta}$  on the entire probability space, we formally set $\hat{\beta}=0$ if $\bar{K}\le 1/\delta$ and $\hat{\beta}=1$ if $\bar{K}\ge T/\delta$. As a result, $\hat{\beta}$ can be expressed as
\begin{align}\label{def_g}
\hat{\beta}=g(\bar{K}), \ \ \textrm{where} \ \ 	g(x):=
	\begin{cases}
	0 &\textrm{if} \ x\in [0,1/\delta];\\
	\eta^{-1}(x) &\textrm{if} \ x\in (1/\delta,T/\delta);\\
	1 &\textrm{if} \ x\in [T/\delta,\infty).
	\end{cases}
\end{align}
This allows $\hat{\beta}$ to take 0 and 1 even though the true parameter $\beta$ cannot. However, in practical situations where $T$ is very large, this will not be a serious issue since, as the following theorem shows, the probability $\mathbb{P}(\bar{K}\not\in (1/\delta,T/\delta))$ can be made as small as we want by taking $T$ large enough.

\begin{theorem}\label{theorem_N-bar}
Let $\delta>0$. Let $\bar{K}$ be the sample mean of a random sample $K_1,\ldots,K_n$ of size $n$ from the distribution of $K$ satisfying \eqref{Edelta_N}. Then as $T\to\infty$, 
\begin{align}\label{asymptotic_Kbar_1}
\mathbb{P}\left(\bar{K}\le \frac 1\delta\right)\le n\mathbb{P}(D_{1+\delta}\ge T) \sim \frac{n(1+\delta)}{\Gamma(1-\beta)}T^{-\beta}.
\end{align}
Moreover, there exists a constant $C>0$ not depending on $T$, $\delta$ or $n$ such that for any $T>e^{1-\gamma}$, 
\begin{align}\label{asymptotic_Kbar_2}
\mathbb{P}\left(\bar{K}\ge \frac T\delta\right)
\le n\mathbb{P}(D_T<T)
\le ne^{-C T}.
\end{align}
\end{theorem}

\begin{proof}
If $K_i>1/\delta$ for all $i$, then $\sum_{i=1}^n K_i>n/\delta$, so 
by \eqref{Edelta_N}, \eqref{ineq_Spsi} and the inverse relation between $(D_t)$ and $(E_t)$, 
\begin{align*}
    \mathbb{P}\left(\bar{K}\le \frac 1\delta\right)
    &\le \mathbb{P}\left(\bigcup_{i=1}^n \left\{K_i\le \frac 1\delta\right\}\right)
    \le n\mathbb{P}\left(K\le \frac 1\delta\right)
    =n\mathbb{P}(E^\delta_T\le 1)\\
    &\le 
    n\mathbb{P}(E_T\le 1+\delta)
    =n\mathbb{P}(D_{1+\delta}\ge T).
\end{align*}  
The latter clearly decays to 0 as $T\to \infty$, and the exact rate of decay follows from the asymptotic behavior of the tail probability of $D_t$ given in Section \ref{section_2}, thereby yielding \eqref{asymptotic_Kbar_1}.

On the other hand, by a similar argument together with the fact that $E_T$ has a density (and hence, $\mathbb{P}(E_T=T)=0$), 
\begin{align*}
    \mathbb{P}\left(\bar{K}\ge \frac T\delta\right)
    &\le \mathbb{P}\left(\bigcup_{i=1}^n \left\{K_i\ge \frac T\delta\right\}\right)
    = n\mathbb{P}\left(K\ge \frac T\delta\right)
    =n\mathbb{P}(E^\delta_T\ge T)\\
    &\le n\mathbb{P}(E_T\ge T)
    =n\mathbb{P}(E_T> T) =n\mathbb{P}(D_T< T).
\end{align*}
By Markov's inequality and relation \eqref{Laplace_stable} with $\psi(s)=s^\beta$, for any fixed $\lambda>0$, 
$
    \mathbb{P}(D_T< T)=\mathbb{P}(e^{\lambda(T-D_T)}>1)
    \le \mathbb{E}[e^{\lambda(T-D_T)}]
    =e^{-(\lambda^\beta-\lambda) T}. 
$
Taking $\lambda=1/2$ so that $C:=\lambda^\beta-\lambda>0$ gives \eqref{asymptotic_Kbar_2}. 
\end{proof}

\begin{remark}\label{remark_Tlarge}
\begin{em}
The construction of our estimator for the parameter $\beta$ is completely different from those of the existing estimators discussed in Section \ref{section_1}, except for the one proposed in \cite{Cahoy2012}. Namely, we analyze the distribution of $K$, which represents the ``number'' of the observed constant periods minus 1, rather than the distribution of the ``lengths'' of the constant periods.  
On the other hand, in \cite{Cahoy2012}, the author derived an estimator using the logarithmic moment of the stable random variable $D_1$ together with the equality in distribution of $E_t$ and $t^\beta D_1^{-\beta}$. (The latter equality is expressed as $Z= t^\alpha R^{-\alpha}$ in that paper.) The estimator for $\beta$ is given for each fixed $t$ by $\hat{\beta}_{\textrm{Cahoy}}(t)=(\hat{\mu}_{\log E_t}+\gamma)/(\log t+\gamma)$, where $\hat{\mu}_{\log E_t}$ denotes the sample mean corresponding to the theoretical mean $\mathbb{E}[\log E_t]$. We can observe a connection of Cahoy's estimator at $t=T$ with our method-of-moments-like estimator. Indeed, $\hat{\beta}_{\textrm{Cahoy}}(T)$ is calculated from observed values of the random variable $\log E_T$, which can be approximated by $\log E^\delta_T=\log (K\delta)$ with $\delta$ small; thus, $\hat{\beta}(T)$ is connected with the random variable $K$ arising in the method-of-moments-like estimation. 
 \end{em}
\end{remark}

The following theorem concerns asymptotic properties of the method-of-moments-like estimator $\hat{\beta}$ for large sample size $n$ (i.e.\ when many observations of the same subdiffusion process are available as in the real-life example given in Section \ref{section_5}). In particular, part (c) allows us to determine how large $T$ should be in order to achieve a target level for the asymptotic variance of $\hat{\beta}$, which is important in applications. Recall that the Gamma function is convex on $(0,\infty)$ and attains the minimum $\Gamma_{\min}:=\min_{x>0} \Gamma(x)\approx 0.8856$ at $x\approx 1.4616$.

\begin{theorem}\label{theorem_asymptotic}
Suppose $\delta>0$ and $\beta\in(0,1)$. Let $K_1,K_2,\ldots$ be an i.i.d.\ sequence from the distribution of $K$ satisfying \eqref{Edelta_N}.
For each $n\in\mathbb{N}$, let $\hat{\beta}_n:=g(\bar{K}_n)$ be the estimator defined in \eqref{def_g} with $\bar{K}_n=\sum_{i=1}^n K_i/n$.
\begin{enumerate}
    \item[\emph{(a)}] For any fixed $T>e^{1-\gamma}\approx 1.5262$, 
        $\lim_{n\to\infty} \hat{\beta}_n \le \beta$ a.s.
    \item[\emph{(b)}] 
        $\lim_{T\to\infty}\lim_{n\to\infty} \hat{\beta}_n = \beta$ a.s.
    \item[\emph{(c)}] For any fixed $T>e^{1-\gamma+\sqrt{\pi^2/6}}\approx 5.5032$,
\begin{align}\label{asymptotic_3}
    \limsup_{n\to\infty} \bigl(n \textit{Var}(\hat{\beta}_n)\bigr)
    < \frac{2}{(\log T+\gamma)^2\Gamma(2\beta+1)}
    < \frac{2}{\Gamma_{\min}(\log T+\gamma)^2}. 
\end{align}
\end{enumerate}
\end{theorem}

\begin{proof}
(a) Suppose $T>e^{1-\gamma}$. 
Note that $\mu_K=\mathbb{E}[K]>0$ since $\mathbb{P}(K=0)=\mathbb{P}(D_\delta>T)<1$, whereas
$
	\mu_K\le \eta(\beta)<T/\delta
$
due to \eqref{key_estimate} and Proposition \ref{proposition_monotone}. 
 Thus, $0<\mu_K<T/\delta$. 
By the strong law, $\bar{K}_n \to \mu_K$ as $n\to\infty$ a.s., and since $g$ is continuous, 
\begin{align}\label{LLN}
    \lim_{n\to\infty}\hat{\beta}_n
    =\lim_{n\to\infty}g(\bar{K}_n) 
    =g(\mu_K) \ \ \textrm{a.s.}
\end{align}
If $0<\mu_K\le 1/\delta$, then $g(\mu_K)=0< \beta$. If $1/\delta<\mu_K<T/\delta$, then $g(\mu_K)=\eta^{-1}(\mu_K)\le \eta^{-1}(\eta(\beta))=\beta$ since $\eta^{-1}$ is increasing. 
In both cases, $\lim_{n\to\infty} \hat{\beta}_n=g(\mu_K)\le \beta$ a.s.

(b) By (a), 
it suffices to show that $\liminf_{T\to\infty}g(\mu_K) \ge \beta$. 
Note that $\mu_K\to \infty$ as $T\to\infty$ due to \eqref{key_estimate}.
Note also that
\[
    \underline{\eta}(\beta):=\frac{T^\beta}{\delta}< \eta(\beta)\le \frac{T^\beta}{\delta \Gamma_{\min}} =: \overline{\eta}(\beta). 
\]
Take $T$ large enough so that $\mu_K>1/\delta$ and $\underline{\eta}(\beta)-1>1/(\delta \Gamma_{\min})$. Then by \eqref{key_estimate}, it follows that 
\[
    g(\mu_K)
    =\eta^{-1}(\mu_K)
    \ge \eta^{-1}(\eta(\beta)-1)
    > \eta^{-1}(\underline{\eta}(\beta)-1).
\]
Clearly, $\overline{\eta}(\beta)$ is a strictly increasing bijection from $(0,1)$ to $(1/(\delta \Gamma_{\min}), T/(\delta \Gamma_{\min}))$ with inverse  $\overline{\eta}^{-1}(x)=\log(\delta \Gamma_{\min} x)/\log T$, and $\eta^{-1}(x)\ge \overline{\eta}^{-1}(x)$ for all $x\in (1/(\delta \Gamma_{\min}), T/\delta)$. Thus, 
\begin{align}\label{g_mu}
     g(\mu_K)
     > \overline{\eta}^{-1}(\underline{\eta}(\beta)-1)
     =\frac{\log[\delta\Gamma_{\min}(T^\beta/\delta-1)]}{\log T}
     =\frac{\log \Gamma_{\min}+\log(T^\beta-\delta)}{\log T},
\end{align}
from which the inequality $\liminf_{T\to\infty}g(\mu_K) \ge \beta$ follows, as desired.

(c) By \eqref{Edelta_N} and \eqref{first_moment} with $p=2$,
\begin{align}\label{second_moment}
	\mathbb{E}[K^2]=\frac{1}{\delta^2}\mathbb{E}[(E^\delta_T)^2]
	\le \frac{1}{\delta^2}\mathbb{E}[(E_T)^2]
	=\frac{2T^{2\beta}}{\delta^2 \Gamma(2\beta+1)}. 
\end{align}
In particular, $K_1,K_2,\ldots$ are i.i.d.\ with finite second moment, so by the central limit theorem, as $n\to \infty,$ 
$
    \sqrt{n}(\bar{K}_n-\mu_K) \longrightarrow^{\mathcal{D}} Z\sim \mathcal{N}(0,\sigma_K^2),
$
where $\sigma_K^2$ is the theoretical variance of $K$. 
We wish to apply the delta method to $\hat{\beta}_n=g(\bar{K}_n)$, but since $g'(1/\delta)$ does not exist and any integer-order derivative of $g$ vanishes on $(0,1/\delta)$, the delta method fails if $0<\mu_K\le 1/\delta$, even using higher order Taylor approximations.

To apply the delta method without any technical issues, we smooth out the function $g(x)$ at $x=1/\delta$ as follows: 
\[
	g_\varepsilon(x):=
	\begin{cases}
	h_\varepsilon(x)  &\textrm{if} \ x\in [0,1/\delta];\\
	g(x) &\textrm{if} \ x\in (1/\delta,\infty),
	\end{cases}
\]
where $h_\varepsilon:(0,1/\delta]\to (-\varepsilon,0]$ with small $\varepsilon>0$ is a smooth, strictly increasing, convex function such that 
$
	D_{-}h_\varepsilon(1/\delta)
	=D_{+}g(1/\delta),
$
with $D_{-}$ and $D_{+}$ denoting the derivatives from the left and right, respectively. 
 Then $g_\varepsilon$ is a smooth, strictly increasing function on $(0,T/\delta)$, and hence, $g_\varepsilon'(\mu_K)$ exists and is positive regardless of the value of $\mu_K\in(0,T/\delta)$.
 Therefore, for the modified estimator $\hat{\beta}_{n,\varepsilon}$ defined by 
$
	\hat{\beta}_{n,\varepsilon}=g_\varepsilon(\bar{K}_n), 
$
by the delta method, as $n\to\infty$,  
$
	\sqrt{n}(\hat{\beta}_{n,\varepsilon}-g_\varepsilon(\mu_K)) 
	\longrightarrow^{\mathcal{D}} g_\varepsilon'(\mu_K) Z\sim \mathcal{N}\bigl(0,\sigma_K^2 [g_\varepsilon'(\mu_K)]^2\bigr),
$
 and in particular, 
\begin{align}\label{variance_limit}
    \lim_{n\to\infty}\bigl(n \textit{Var}(\hat{\beta}_{n,\varepsilon}) \bigr)
    =\sigma_K^2 [g_\varepsilon'(\mu_K)]^2.
\end{align}
Moreover, since $g_\varepsilon$ is convex on $(0,1/\delta)$ and concave on $(1/\delta,T/\delta)$ when $T>e^{1-\gamma+\sqrt{\pi^2/6}}$ due to Proposition \ref{proposition_monotone} (as the convexity of $\eta$ implies the concavity of $\eta^{-1}$), it follows that 
 \[
    0< g_\varepsilon'(\mu_K)
    \le \max_{x\in[0,\,T/\delta)}g_\varepsilon'(x)
    =g_\varepsilon'(1/\delta)
    =\frac{1}{D_{+}\eta(0)}
    =\frac{\delta}{T^\beta(\log T+\gamma)}, 
 \]
 where the last equality follows from \eqref{eta_prime}. 
Combining this with \eqref{second_moment} gives 
\begin{align}\label{variance_estimate}
    \sigma_K^2[g_\varepsilon'(\mu_K)]^2
    < \frac{\delta^2\mathbb{E}[K^2]}{T^{2\beta}(\log T+\gamma)^2}
	\le \frac{2}{(\log T+\gamma)^2\Gamma(2\beta+1)}
	< \frac{2}{\Gamma_{\min}(\log T+\gamma)^2}. 
\end{align}
Finally, note that $\hat{\beta}_{n,\varepsilon}$ takes values in $(-\varepsilon,1]$ and that $\hat{\beta}_n=\hat{\beta}_{n,\varepsilon}\mathbf{1}_{\{\hat{\beta}_{n,\varepsilon}>0\}}$, where $\mathbf{1}_A$ is the indicator function of a set $A$. This implies 
    $\textit{Var}(\hat{\beta}_n)< \textit{Var}(\hat{\beta}_{n,\varepsilon})$ for each $n$.
Putting the latter together with \eqref{variance_limit} and  \eqref{variance_estimate} gives \eqref{asymptotic_3}.
\end{proof}

\begin{remark}\label{remark_unbiasedness}
\begin{em}
(a) \textbf{(Asymptotic unbiasedness and consistency)} Theorem \ref{theorem_asymptotic}(b) shows that $\hat{\beta}_n$, when regarded as an estimator indexed by both $n$ and $T$, is asymptotically unbiased and consistent as the indices go to infinity.

(b) By \eqref{LLN} and \eqref{g_mu}, for large $T$ and for each path, 
the estimation error $\beta - \hat{\beta}_n$ with sufficiently large $n$ satisfies
\begin{align}\label{revision_2}
    \beta - \hat{\beta}_n<\beta-\frac{\log \Gamma_{\min}+\log(T^\beta-\delta)}{\log T}
    =-\frac{\log[\Gamma_{\min}(1-\delta T^{-\beta})]}{\log T}. 
\end{align} 
\end{em}
\end{remark}

\begin{remark}\label{remark_mixture}
\begin{em}
\textbf{(Robustness)} Our estimation method can be applied to a more general subordinator whose Laplace exponent $\psi(s)$ in \eqref{Laplace_stable} has the asymptotic behavior  
\begin{align}\label{robust_1}
    \psi(s)\sim s^\beta \ \ \textrm{as} \ s\downarrow 0.
\end{align}
(The special case when $\psi(s)=s^\beta$ for all $s>0$ recovers a $\beta$-stable subordinator.)
 Indeed, in that case, the Laplace transform of the function $T\mapsto \mathbb{E}[E_T]$ satisfies 
$
   \int_0^\infty \mathbb{E}[E_T]e^{-sT}\,\textrm{d}T
    = 1/[s\psi(s)]\sim 1/s^{\beta+1}
$
as $s\downarrow 0$ (see e.g.\ Proposition 3.1 in \cite{JumKobayashi}), and hence, by the Tauberian theorem (see e.g.\ Section 1.7 of \cite{Bingham_book}), $\mathbb{E}[E_T]\sim T^\beta/\Gamma(\beta+1)$ as $T\to \infty$. In other words, equality \eqref{first_moment} with $p=1$ approximately holds for large enough $T$. Hence, when $T$ is large, it is reasonable to use the estimator defined in \eqref{def_g} to estimate the value of $\beta$. 

Note that the asymptotic condition \eqref{robust_1} means the subordinator has power law probability tails. Indeed, by the proof of Lemma 3.4 in \cite{JumKobayashi}, for each fixed $t>0$,
$
    \int_0^\infty \mathbb{P}(D_t\ge x) e^{-sx}\,\textrm{d}x 
    = (1-e^{-t\psi(s)})/s,
$
which is asymptotically equivalent to $t/s^{1-\beta}$ as $s\downarrow 0$ if and only if condition \eqref{robust_1} holds, but by the Tauberian theorem, the latter is equivalent to the statement that $\mathbb{P}(D_t\ge x)\sim t x^{-\beta}/\Gamma(1-\beta)$ as $x\to \infty$. 
\end{em}
\end{remark}

\section{Numerical Comparison to Existing Methods}
\label{section_4}

In this section, we use simulations in Matlab to compare the performance of the method-of-moments-like estimator $\hat{\beta}$ defined in \eqref{def_g} (MOM-like estimator for short) to Cahoy's estimator in \cite{Cahoy2012}, Hill's estimator in \cite{Hill75}, and the Meerschaert--Scheffler estimator in \cite{MS98} (MS estimator for short). We chose the latter three estimators for comparison purposes since i) they are simple to apply, ii) Cahoy's estimator is also constructed via the method of moments, iii) Hill's estimator has been widely employed in practice, and iv) both Cahoy's and MS estimators rely on all the data 
like the MOM-like estimator.
We calculate Hill's estimator based on the largest 10\% of the data as that is common in practice.

Note that this section focuses on data that are realizations of the discretized inverse $\beta$-stable subordinator $(E^\delta_t)$ but not on data coming from a time-changed process of the form $(Y_{E^\delta_t})$. This is because under some assumptions on the outer process $(Y_t)$, a procedure for estimating $\beta$ for data following $(Y_{E^\delta_t})$ is indeed the same as the procedure for estimating $\beta$ for data following $(E^\delta_t)$. We postpone a detailed discussion of this matter to Section \ref{section_5}, where we treat real data observable by means of a time-changed process.

We first generate $n$ paths of the discretized time change $(E^\delta_t)$ on a fixed time interval $[0,T]$, calculate the sample mean $\bar{K}$ for the observed paths, and obtain the MOM-like estimate via equation \eqref{def_g}, 
where $K_i$ is determined as the number of constant periods minus 1 for the $i$th path. 
On the other hand, as mentioned in Remark \ref{remark_Tlarge}, Cahoy's estimator requires realizations of the time change $(E_t)$; however, only the paths of the discretized time change $(E^\delta_t)$ are available here. Therefore, instead of $\hat{\mu}_{\log E_T}$ appearing in Remark \ref{remark_Tlarge}, we use $\hat{\mu}_{\log E^\delta_T}=\hat{\mu}_{\log (K\delta)}:=\frac 1n\sum_{i=1}^n \log(K_i\delta)$ to obtain an approximate version of Cahoy's estimate.
Using the same data set but after aggregating all the constant periods observed in the $n$ paths, we calculate Hill's and MS estimates based on the ``lengths'' (rather than the ``number'') of the constant periods via the formulas provided in \cite{Janczura}.

Note that our simulation 
relies on the choice of three hyper-parameters --- the step size $\delta$, the final time $T$, and the number $n$ of observed paths. So we simulate using various combinations of the hyper-parameters over different values of $\beta$. 
 The results in three particular cases with $\delta$ fixed to be 1 are summarized in the tables in Figure \ref{fig:compare_1}, which indicates that the MOM-like estimator i) improves its performance as $T$ and $n$ increase regardless of the value of $\beta$, ii) is more accurate than Hill's and MS estimators for large values of $\beta$, and iii) performs better than the approximate version of Cahoy's estimator for small values of $\beta$.

\begin{figure}\centering\tiny
      \begin{tabular}{|c|c|c|c|c|}\hline
         \multicolumn{5}{|c|}{$\delta=1$, $T=10000$, $n=3000$}\\ \hline
		$\beta$ & MOM & Cahoy & Hill & MS \\ \hline
		0.1 & 0.1176 & 0.1470 & 0.0965 & 0.1034\\
		0.2 & 0.2070 & 0.2233 & 0.1921 & 0.2059 \\
		0.3 & 0.3011 & 0.3105 & 0.2994 & 0.3265\\
		0.4 & 0.4001 & 0.4021 & 0.3926 & 0.4262\\
		0.5 & 0.5001 & 0.4999 & 0.4977 & 0.5369\\
		0.6 & 0.5997 & 0.5998 & 0.6100 & 0.7130\\
		0.7 & 0.6995 & 0.7012 & 0.7348 & 0.7215\\
		0.8 & 0.7999 & 0.8002 & 0.9000 & 0.8963\\
		0.9 & 0.8999 & 0.9001 & 1.1988 & 0.9586\\ \hline
	\end{tabular}
	\hspace{1mm}
      \begin{tabular}{|c|c|c|c|c|}\hline
         \multicolumn{5}{|c|}{$\delta=1$, $T=10000$, $n=100$}\\ \hline
		$\beta$ & MOM & Cahoy & Hill & MS \\ \hline
		0.1 & 0.1263 & 0.1574 & 0.1071 & 0.1306\\
		0.2 & 0.2103 & 0.2313 & 0.1803 & 0.2215\\
		0.3 & 0.3031 & 0.3137 & 0.2839 & 0.3142\\
		0.4 & 0.4020 & 0.4074 & 0.3940 & 0.3806\\
		0.5 & 0.5051 & 0.5118 & 0.4915 & 0.5585\\
		0.6 & 0.6037 & 0.5987 & 0.6136 & 0.6236\\
		0.7 & 0.6987 & 0.7024 & 0.7418 & 0.8323\\
		0.8 & 0.8004 & 0.8034 & 0.8950 & 0.8163\\
		0.9 & 0.9060 & 0.9056 & 1.2084 & 0.9797\\ \hline
	\end{tabular}
	\hspace{1mm}
      \begin{tabular}{|c|c|c|c|c|}\hline
         \multicolumn{5}{|c|}{$\delta=1$, $T=1000$, $n=100$}\\ \hline
		$\beta$ & MOM & Cahoy & Hill & MS \\ \hline
		0.1 & 0.1346 & 0.1791 & 0.0852 & 0.0839\\
		0.2 & 0.2141 & 0.2409 & 0.1989 & 0.1829 \\
		0.3 & 0.3105 & 0.3338 & 0.2956 & 0.3542\\
		0.4 & 0.3894 & 0.3898 & 0.4104 & 0.4591\\
		0.5 & 0.5186 & 0.5140 & 0.5137 & 0.6280\\
		0.6 & 0.5749 & 0.5759 & 0.5714 & 0.7811\\
		0.7 & 0.6932 & 0.6880 & 0.7094 & 0.7411\\
		0.8 & 0.7921 & 0.7965 & 0.8732 & 0.8927\\
		0.9 & 0.8911 & 0.8940 & 1.1757 & 0.9452\\ \hline
	\end{tabular}
  \caption{Comparison of MOM-like estimates with Cahoy's, Hill's and MS estimates.}
  \label{fig:compare_1}
\end{figure}

\begin{figure}\centering\scriptsize
      \begin{tabular}{|c|c|c|c|c|c|c|c|c|}\hline
		 & \multicolumn{2}{|c|}{MOM} & \multicolumn{2}{|c|}{Cahoy} & \multicolumn{2}{|c|}{Hill} & \multicolumn{2}{|c|}{MS}\\ \hline
		$\beta$ & mean & variance & mean & variance & mean & variance & mean & variance\\ \hline
		0.1 & 0.1160 & 0.0001 & 0.1431 & 0.0001 & 0.1071 & 0.0005 & 0.1133 & 0.0004 \\
		0.2 & 0.2045 & 0.0002 & 0.2189 & 0.0003 & 0.1991 & 0.0011 & 0.2128 & 0.0012 \\
		0.3 & 0.3016 & 0.0002 & 0.3091 & 0.0002 & 0.2935 & 0.0008 & 0.3097 & 0.0018 \\
		0.4 & 0.4011 & 0.0001 & 0.4059 & 0.0002 & 0.3968 & 0.0004 & 0.4126 & 0.0031\\
		0.5 & 0.5003 & 0.0001 & 0.5035 & 0.0002 & 0.4988 & 0.0003 & 0.5433 & 0.0035\\
		0.6 & 0.5990 & 0.0001 & 0.6018 & 0.0001 & 0.6095 & 0.0003 & 0.6322 & 0.0038\\
		0.7 & 0.6998 & 0.0001 & 0.7006 & 0.0002 & 0.7376 & 0.0001 & 0.7656 & 0.0056\\
		0.8 & 0.8002 & 0.0000 & 0.8013 & 0.0001 & 0.8992 & 0.0001 & 0.8919 & 0.0066\\
		0.9 & 0.8992 & 0.0000 & 0.8987 & 0.0001 & 1.1981 & 0.0001 & 1.0321 & 0.0096\\ \hline
	\end{tabular}
  \caption{Comparison of the sample means and sample variances of MOM-like estimates with those of Cahoy's, Hill's and MS estimates based on 100 repetitions of estimation, where $\delta=1$, $T=23400$ and $n=44$.}
  \label{fig:compare_2}
\end{figure}

Next, with $\delta=1$ still fixed, we take $T=23400$ and $n=44$, which provide the setting for a real data to be treated in Section \ref{section_5}, and repeat the above estimation procedure 100 times to produce 100 estimates based on each of the four different methods. We then calculate the sample mean and sample variance for each method. The results summarized in Figure \ref{fig:compare_2} show that the MOM-like estimator gives a reasonably accurate mean with a very low variance for the entire range of $\beta$ values.

We now turn our attention to the effect of the value of $\delta>0$ on the MOM-like estimation. 
It is natural to expect that the accuracy of the MOM-like estimation increases as $\delta$ decreases, which is also indicated by the upper bound for the pathwise estimation error in \eqref{revision_2}. Here, we again record the sample means of estimates based on 100 repetitions of the MOM-like estimation with $T=23400$ and $n=44$ fixed but this time with different values of $\delta$. Figure \ref{fig:delta} gives the simulation results with $\delta$ chosen from the four values $\{0.1, 0.7, 1.3, 1.9\}$, where the corresponding sample variances are all within 0.00025. It shows that when the true $\beta$ value is small, the MOM-like estimator performs better with a smaller value of $\delta$;
 however, for large $\beta$, choosing a smaller value of $\delta$ from the particular four values contributes little to improving the estimation results. 
Unfortunately, we do not have a straightforward criterion for choosing the value of $\delta$; it depends on the accuracy level that one wants to achieve, and generally speaking, a small $\delta$ value is recommended when the data exhibits long constant periods (which implies the true $\beta$ value is small). 
On the other hand, we suggest taking $\delta$ with $T^\beta>\delta$ for any possible $\beta\in(0,1)$ so that the quantity $\eta(\beta)-1$ appearing in the fundamental estimates in \eqref{key_estimate} is guaranteed positive and hence provides a meaningful lower bound for $\mu_K$.  
For the latter purpose, taking $\delta\le 1$ suffices since we always assume $T>e^{1-\gamma}>1$. 
In Section 5, where we deal with real data with $T=23400$ and $n=44$, we take $\delta=1$ since Figures \ref{fig:compare_2} and \ref{fig:delta} guarantee satisfactory performance of the MOM-like estimator for these values of $T$ and $n$ with $\delta=1$.

\begin{figure}\centering\scriptsize
       \begin{tabular}{|c|c|c|c|c|}\hline
         \multicolumn{5}{|c|}{MOM, $T=23400$, $n=44$}\\ \hline
		$\beta$ & $\delta=0.1$ & $\delta=0.7$ & $\delta=1.3$ & $\delta=1.9$ \\ \hline
		0.1 & 0.1006 & 0.1115 & 0.1218 &0.1285\\
		0.2 & 0.1988 & 0.2042 & 0.2064 &0.2108\\
		0.3 & 0.3008 & 0.2995 & 0.3035 &0.3029\\
		0.4 & 0.3997 & 0.3972 & 0.4001 &0.4027\\
		0.5 & 0.4990 & 0.5008 & 0.5011 &0.4992\\
		0.6 & 0.5997 & 0.6001 & 0.5992 &0.6004\\
		0.7 & 0.7005 & 0.6998 & 0.7003 &0.6977\\
		0.8 & 0.7987 & 0.7993 & 0.8001 &0.7996\\
		0.9 & 0.8996 & 0.9004 & 0.9002 &0.8988\\ \hline
	\end{tabular}
  \caption{Comparison of the sample means of MOM-like estimates based on 100 repetitions of estimation with different values of $\delta>0$.}
  \label{fig:delta}
\end{figure}

\begin{remark}\label{remark_advantage}
\begin{em}
(a) The above simulations may seem to suggest that, for realizations of $(E^\delta_t)$,  
the MOM-like estimator outperforms the approximate version of Cahoy's estimator when the true $\beta$ is small. However, Cahoy's estimator has the following advantages: i) it does not involve a root finding algorithm and ii) it is valid for any fixed $T$ and comes with explicit formulas for confidence intervals (whereas the MOM-like estimator requires equation \eqref{estimator_def} to be solved with a certain algorithm and can only provide approximate confidence intervals for large $T$ via the normal approximation depending on small $\varepsilon$ in the proof of Theorem \ref{theorem_asymptotic}(c)). On the other hand, \textit{when a decent computing environment is available and $T$ is very large as is often the case with real data, the MOM-like estimator may become a suitable option as it allows one to estimate the entire range of $\beta$ values with reasonable accuracy and precision.}

(b) In \cite{Wylomanska17_modifiedCDF}, the authors 
introduced a modified version of the cumulative distribution function (CDF) for the lengths of constant periods. The modification accounts for the fact that the beginning and ending of a given constant period in real data may have actually occurred at time points when the data was not recorded. The latter is an important issue when parameter estimation is carried out based on the ``lengths'' of constant periods, while it becomes less of an issue with the MOM-like or Cahoy's estimation since the ``number'' of constant periods 
is not affected by the exact timing of the beginning and ending of each constant period. Being able to avoid a discussion of this subtle issue as well as the simple formula for finding the estimate 
is an advantage of using the MOM-like or Cahoy's estimation. 
In the above discussion, we did not compare the modified CDF method to the other four estimation methods since (i) it requires much more computing power and (ii) it considers the setting for real data observed at pre-specified discrete time points (while we used simulated paths of $(E^\delta_t)$ in which the beginning and ending of each constant period may occur at any time point in $[0,T]$). 
\end{em}
\end{remark}

\section{Real-Life Application: Low-Volume Stock Modeling}
\label{section_5}

This section illustrates how to apply the MOM-like estimator to real data collected from a stock market. Traditionally, low-volume stocks have been difficult to model since they often feature periods in which no trades are made. As a result, the price often has long constant periods throughout the trading day, and therefore, 
it is a good candidate to model using a time-changed process $(X_{E_t})$. However, since an observed data is always discrete, we regard it as a realization of the discretized process $(X_{E^\delta_t})$. 
In terms of the discretized time change $(E^\delta_t)$, we assume that the underlying subordinator $(D_t)$ has power law probability tails with index $\beta\in(0,1)$ even though the use of an exponentially tempered power law might be more appropriate, as pointed out e.g.\ in \cite{MeerschaertRoyShao,OrzelWylomanska11}. (Note that our purpose here is simply to illustrate how to apply our estimation method to given data exhibiting constant periods.)
On the other hand, $(X_t)$ is assumed to be a geometric Brownian motion with representation $X_t=X_0e^{\mu t+\sigma B_t}$, where $(B_t)$ is a Brownian motion independent of the subordinator $(D_t)$ and the constants $\mu \in\mathbb{R}$ and $\sigma>0$ are additional parameters to be estimated. We focus on the logarithmic stock price $(Y_{E_t})$, where $Y_t:=\log (X_t/X_0)=\mu t+\sigma B_t$, and note that the independence assumption implies that $(Y_t)$ is independent of $(E^\delta_t)$.

\begin{figure}\centering\footnotesize
  \begin{minipage}{0.39\hsize}
      \includegraphics[height=4.5cm,width=6cm]{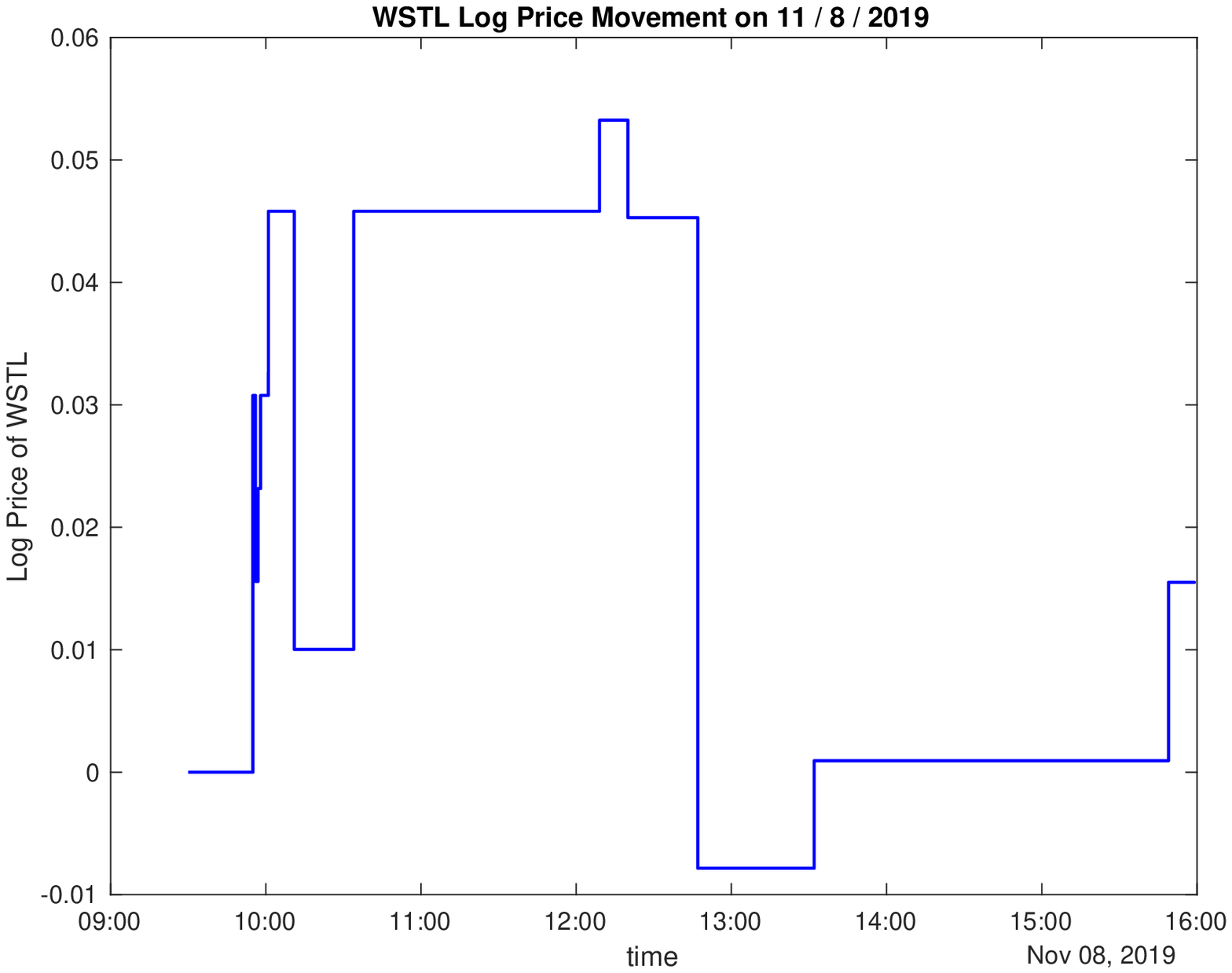}
  \caption{Fluctuation of the logarithmic value of the WSTL price on November 8, 2019.}
  \label{fig:MLAB}
  \end{minipage}
  \hspace{10mm}
   \begin{minipage}{0.38\hsize}
      \includegraphics[height=4.5cm,width=5.7cm]{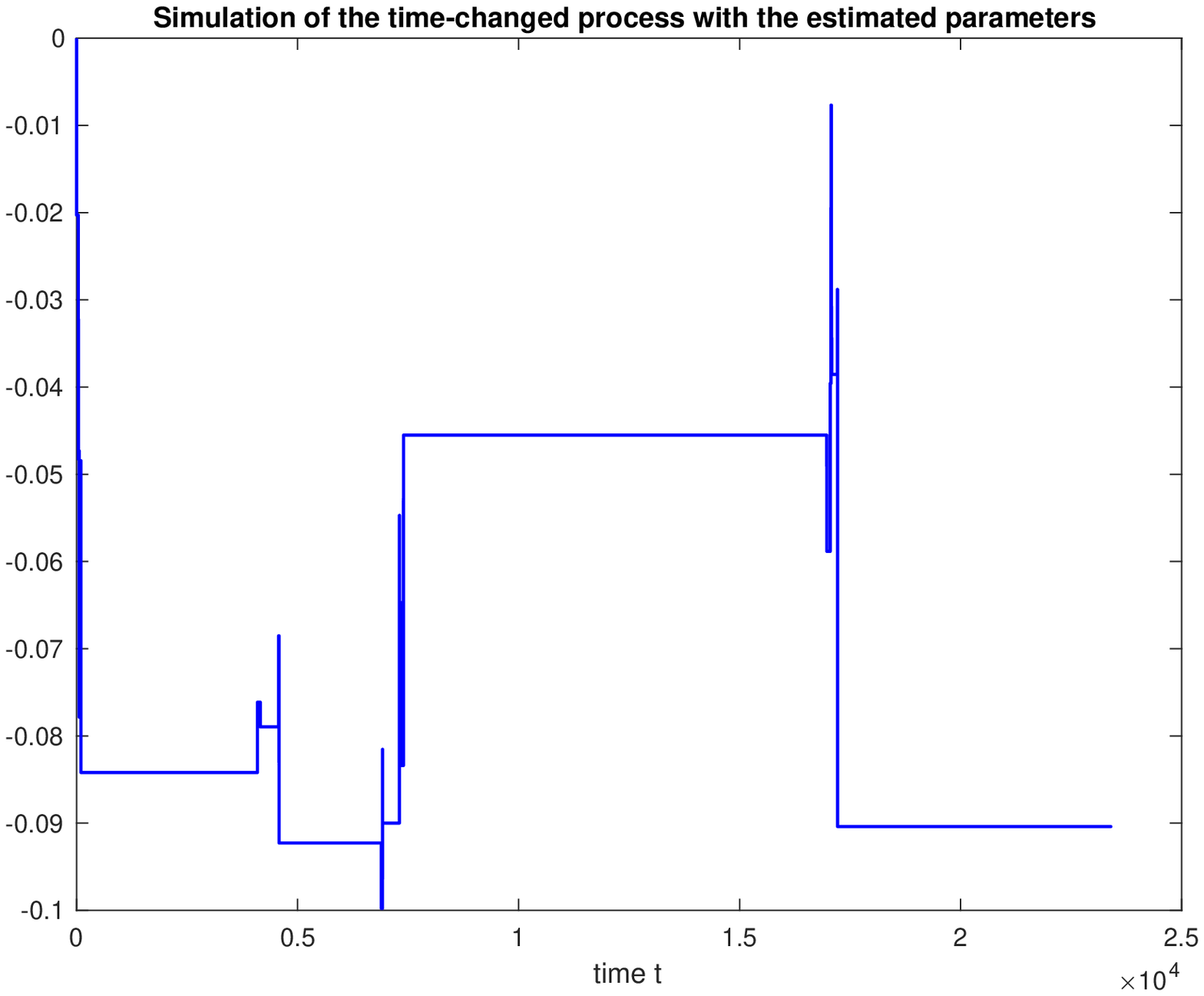}
  \caption{A simulated path of $(Y_{E^\delta_t})$  on $[0,T]$ with $T=23400$ based on the obtained estimates $\hat{\beta}$, $\hat{\mu}$ and $\hat{\sigma}$.}
  \label{fig:simulatedMLAB}
  \end{minipage}
  \end{figure}

Here, we use second-by-second trading data of the low-volume NASDAQ stock with ticker WSTL for the nine weeks starting from September 30, 2019 (consisting of $44$ trading days), obtained from Bloomberg Terminal.
Figure \ref{fig:MLAB} provides the observed path of the logarithmic price on a day, which clearly exhibits some constant periods.
We assume that each of the $n=44$ observed paths is a realization of the discretized process $(Y_{E^\delta_t})$ with $\delta=1$ on the time interval $[0,T]$, with $T = 23400$ being the number of seconds in a trading day.  Recall from Figures \ref{fig:compare_2} and \ref{fig:delta} that the MOM-like estimator performs well for these hyper-parameters if the data is observed by means of $(E^\delta_t)$.

By construction of $(E^\delta_t)$, each path of $(E^\delta_t)$ takes values $0,1,2,\ldots,K$ and has $(K+1)$ constant periods. We claim that the number $N$ of constant periods observed in a path of $(Y_{E^\delta_t})$ equals $K+1$ a.s. To see this, recall that $(Y_t)$ is independent of $(E^\delta_t)$ 
and note that, given the information of the path of $(E^\delta_t)$, $N<K+1$ if and only if $Y_i=Y_{i+1}$ for at least one $i\le K-1$. (For example, if $Y_2=Y_3=Y_4=0$ and the remaining $Y_i$'s are nonzero and distinct from one another, then $N=(K+1)-2$.) Moreover, $Y_0=0$ and the random vector $(Y_1,Y_2,\ldots,Y_K)$ given $(E^\delta_t)$ is non-degenerate multivariate Gaussian and hence has a continuous distribution, so 
\[
	\mathbb{P}(N<K+1)
	=\mathbb{E}\left[\mathbb{P}(N<K+1\,|\,(E^\delta_t))\right]
	=\mathbb{E}\left[\mathbb{P}(Y_i=Y_{i+1} \ \textrm{for at least one} \ i\le K-1\,|\,(E^\delta_t))\right]
	=0, 
\]
which yields $N\ge K+1$ a.s. Since $N>K+1$ is impossible, 
 $N=K+1$ a.s. In other words, \textit{a sample path of $(Y_{E^\delta_t})$ is a step function with constant periods that are completely ascribed to the constant periods of the corresponding path of $(E^\delta_t)$.}

With this observation in mind, we estimate the three parameters $\beta$, $\mu$ and $\sigma$ as follows. 
If an observed path of $(Y_{E^\delta_t})$ constantly changes in value (i.e.\ the values at any two successive time points are distinct), then we consider the number of constant periods to be $T$  (i.e.\ $N=T$). In the other extreme case when an observed path stays constant over the entire time period, we set the number of constant periods to be $1$ (i.e.\ $N=1$). Due to our observation in the previous paragraph, the obtained number $N$ minus 1 for each of the $44$ days is considered a realization of the random variable $K$, and consequently, we obtain a total of $n=44$ realizations of $K$. The extremely large $T$ value makes the observed value of $\bar{K}$ fall in the interval $(1/\delta, T/\delta)=(1,23400)$, and equation \eqref{estimator_def} gives $\hat{\beta}\approx 0.2718$ as the MOM-like estimate for $\beta$. 
(For comparison, Cahoy's, Hill's and MS estimates are 0.2950, 0.7501 and 0.6605, respectively. 
The discrepancy between the estimated values indicates the power law distribution may not be an appropriate model here.)

Next, following the idea presented in e.g.\ \cite{OrzelWylomanska11}, we remove all the constant periods from each of the 44 observed paths and record all the jump sizes over the 44 days. Since each jump size is given in the form $Y_{i+1}-Y_{i}$ with $Y_t=\mu t+\sigma B_t$, we regard the recorded jump sizes as a random sample drawn from $\mu+\sigma Z$, where $Z\sim \mathcal{N}(0,1)$. 
The obtained sample gives the sample mean $\hat{\mu}\approx -4.901\times 10^{-6}$ and sample standard deviation $\hat{\sigma}\approx 1.789\times 10^{-2}$.
Figure \ref{fig:simulatedMLAB} provides a simulated path of the time-changed process $(Y_{E^\delta_t})$ with the obtained estimates $\hat{\beta}$, $\hat{\mu}$ and $\hat{\sigma}$, as a comparison to Figure \ref{fig:MLAB}.

\begin{remark}
\begin{em}
(a) The fact that the outer process $(Y_t)$ is independent of the discretized time change $(E^\delta_t)$ and has a continuous finite-dimensional distribution is essential in the above discussion since otherwise the number $N$ of constant periods in $(Y_{E^\delta_t})$ would not necessarily coincide with the number $K+1$ of constant periods in $(E^\delta_t)$. In particular, our approach for parameter estimation does not seem to be directly applicable to a fractional Poisson process, which is a renewal process with Mittag--Leffler waiting times, or equivalently, a Poisson process time-changed by an independent inverse stable subordinator; see \cite{MNV3} for the two equivalent formulations of the process. With the latter formulation, the outer process (i.e.\ the Poisson process) has a discrete distribution, and the intensity parameter $\lambda$ of the Poisson process as well as the index $\beta$ of the underlying subordinator affects the distribution of the constant periods; see e.g.\ \cite{Cahoy2010} for estimation of the parameters $\lambda$ and $\beta$. 

(b) Due to the simulation results in Figure \ref{fig:delta}, with the above values of $n$ and $T$, the choice of $\delta=1$ seems reasonable for finding the MOM-like estimate for $\beta$. 
However, the value of $\delta$ has a notable effect on the estimates for $\mu$ and $\sigma$ ascribed to the outer process $(Y_t)$.
 Indeed, if $\delta<1$ is chosen, then since $(E^\delta_t)$ takes values $0,\delta,2\delta,\ldots, K\delta$, the time-changed process $(Y_{E^\delta_t})$ takes values $Y_0, Y_\delta, Y_{2\delta},\ldots, Y_{K\delta}$, so the jump sizes are of the form  $Y_{(i+1)\delta}-Y_{i\delta}$. Therefore, it seems appropriate to regard the recorded jump sizes as a random sample from $\mu \delta+\sigma Z_\delta$, where $Z_\delta \sim \mathcal{N}(0, \delta)$. This implies the estimates $\hat{\mu}$ and $\hat{\sigma}$ obtained above with $\delta=1$ must be replaced by
$\hat{\mu}_\delta=\hat{\mu}/\delta$ and $\hat{\sigma}_\delta=\hat{\sigma}/\sqrt{\delta}$, respectively, 
  consequently yielding an estimated process $(Y_{E^\delta_t})$ that is different from the one with $\delta=1$.  
In existing literature concerning estimation of the index $\beta$ in the context of subdiffusions, it seems common to assume that the jump sizes of a time-changed process are given in the form $Y_{i+1}-Y_i$ with the time difference being exactly 1 (see e.g.\ \cite{OrzelWylomanska11}). Our choice of $\delta=1$ for the above real data parallels such an interpretation of the jump sizes, while guaranteeing reasonably accurate and precise estimates for the full range of $\beta$ values as discussed in Section \ref{section_4}.
\end{em}
\end{remark}

\vspace{3mm}

\noindent
\textbf{Acknowledgements:} 
The authors are grateful to the three anonymous referees for their valuable comments and suggestions which significantly improved both the contents and presentation of the paper.
Kei Kobayashi's research was partially supported by a Faculty Fellowship at Fordham University. Part of this research was conducted while Phillip Kerger was affiliated with Fordham University. The authors thank Fordham University for their support.

  \bibliographystyle{plain} 
\bibliography{KobayashiK}

\end{document}